\documentclass[11pt]{article}

\usepackage{amsmath, amsthm}
\usepackage{amssymb}
\usepackage{eucal}
\usepackage{hyperref} 
\usepackage{mathrsfs}

\newtheorem{theorem}{Theorem}[section]
\newtheorem{lemma}[theorem]{Lemma}

\theoremstyle{definition}
\newtheorem{example}[theorem]{Example}

\begin{document}
	
	\baselineskip 16pt

	\title{A result on $s$-semipermutable subgroups of finite groups and some applications}
	
	\author{Fawaz Aseeri\\
		{\small Department of Mathematical Sciences, College of Applied Sciences, Umm Al-Qura University,}\\ {\small Makkah 21955, Saudi Arabia}\\
		{\small E-mail:
			fiaseeri@uqu.edu.sa}\\ \\
		{Julian Kaspczyk\footnote{Corresponding author}}\\
		{\small Institut für Algebra, Fakultät Mathematik, Technische Universität Dresden,}\\
		{\small 01069 Dresden, Germany}\\
		{\small E-mail: julian.kaspczyk@gmail.com}}
	
	\date{}
	\maketitle

	\begin{abstract} 
		Let $p$ be a prime number, $G$ be a $p$-solvable finite group and $P$ be a Sylow $p$-subgroup of $G$. We prove that $G$ is $p$-supersolvable if $N_G(P)$ is $p$-supersolvable and if there is a subgroup $H$ of $P$ with $P' \le H \le \Phi(P)$ such that $H$ is $s$-semipermutable in $G$. As applications, we simplify the proofs of some known results and also generalize some known results. 
	\end{abstract}
	

	\let\thefootnote\thefootnoteorig

	\section{Introduction}

	In this paper, all groups are assumed to be finite. We use standard notation and terminology, see for example \cite{DoerkHawkes} or \cite{Huppert}. Throughout, let $p$ be an arbitrary but fixed prime number.
	
	Recall that two subgroups $H$ and $K$ of a group $G$ are said to \textit{permute} if $HK = KH$, or equivalently if $HK$ is a subgroup of $G$. A subgroup $H$ of a group $G$ is said to be \textit{$s$-permutable} in $G$ if $H$ permutes with every Sylow subgroup of $G$. A subgroup $H$ of a group $G$ is called \textit{$s$-semipermutable} in $G$ if $H$ permutes with every Sylow $q$-subgroup $Q$ of $G$ for all primes $q$ not dividing $|H|$. 
	
	By a result of Wielandt \cite[Kapitel IV, Satz 8.1]{Huppert}, a group $G$ with Sylow $p$-subgroup $P$ is $p$-nilpotent if $P$ is regular and $N_G(P)$ is $p$-nilpotent. By \cite[Kapitel III, Satz 10.2 (a)]{Huppert}, a $p$-group $P$ is regular if the nilpotency class of $P$ is less than $p$, i.e. if $Z_{p-1}(P) = P$. Consequently, a group $G$ with Sylow $p$-subgroup $P$ is $p$-nilpotent if $Z_{p-1}(P) = P$ and $N_G(P)$ is $p$-nilpotent. Recently, Xu and Li \cite{XuLi} generalized this result as follows. 
	
	\begin{theorem}
		\label{p-nilpotence_theorem} 
		(\cite[Theorem 1.3]{XuLi}) Let $G$ be a group of order divisible by $p$, and let $P$ be a Sylow $p$-subgroup of $G$. Suppose that there is a normal subgroup $H$ of $P$ with $H \le \Phi(P)$ and $P/H = Z_{p-1}(P/H)$ such that $H$ is $s$-semipermutable in $G$. Suppose moreover that $N_G(P)$ is $p$-nilpotent. Then $G$ is $p$-nilpotent. 
	\end{theorem}
	
	The main result of this paper is a version of Theorem \ref{p-nilpotence_theorem} for $p$-supersolvable groups. Our point of departure is the observation that a $p$-solvable group $G$ with Sylow $p$-subgroup $P$ is $p$-supersolvable if $P$ is abelian and $N_G(P)$ is $p$-supersolvable. To see this, it is enough to consider the case $O_{p'}(G) = 1$. Then we have $C_G(O_p(G)) \le O_p(G)$ by \cite[Chapter 6, Theorem 3.2]{Gorenstein}, whence $P = O_p(G)$ as $P$ is abelian. So $G = N_G(O_p(G)) = N_G(P)$ is $p$-supersolvable, as claimed. 
	
	Our main result generalizes this observation as follows. 
	
	\begin{theorem}
		\label{main_theorem}
		Let $G$ be a $p$-solvable group, and let $P$ be a Sylow $p$-subgroup of $G$. Suppose that there is a subgroup $H$ of $P$ such that $P' \le H \le \Phi(P)$ and such that $H$ is $s$-semipermutable in $G$. Suppose moreover that $N_G(P)$ is $p$-supersolvable. Then $G$ is $p$-supersolvable. 
	\end{theorem}
	
	
	In view of Theorems \ref{p-nilpotence_theorem} and \ref{main_theorem}, one might wonder whether the condition $P' \le H \le \Phi(P)$ in Theorem \ref{main_theorem} can be replaced by the condition that $H \trianglelefteq P$, $H \le \Phi(P)$ and $P/H = Z_{p-1}(P/H)$. The following example demonstrates that this is not the case.
	
	\begin{example}
		Let $G$ be the group indexed in GAP \cite{GAP} as \verb|SmallGroup(216,153)|, and let $P$ be a Sylow $3$-subgroup of $G$. Then $G$ is solvable and hence $3$-solvable. Also, $P$ has nilpotency class $2$. So, with $H = 1$, we have that $H \trianglelefteq P$, $H \le \Phi(P)$ and $P/H = Z_2(P/H) = Z_{3-1}(P/H)$. Moreover, $H$ is $s$-semipermutable in $G$, and $N_G(P)$ is $3$-supersolvable. However, $G$ is not $3$-supersolvable. 
	\end{example}
	
	The paper is organized as follows. After collecting some preliminary results in Section \ref{preliminaries}, we will prove Theorem \ref{main_theorem} in Section \ref{proof_main_theorem}. After that, in Section \ref{proof_special_case}, we will apply Theorem \ref{main_theorem} to give a new proof of a special case of Theorem \ref{p-nilpotence_theorem}. Namely, we use Theorem \ref{main_theorem} to show that a group $G$ with Sylow $p$-subgroup $P$ is $p$-nilpotent provided that $N_G(P)$ is $p$-nilpotent and that there is a subgroup $H$ of $P$ with $P' \le H \le \Phi(P)$ such that $H$ is $s$-semipermutable in $G$. Finally, in Section \ref{applications}, we will use Theorems \ref{p-nilpotence_theorem} and \ref{main_theorem} to generalize a result of Liu and Yu \cite{LiuYu} on pronormal subgroups and a result of Chen et al. \cite{Chen} on weakly $\mathscr{H}$-subgroups and to simplify the proofs of these results.

	\section{Preliminaries} 
	\label{preliminaries} 
	
	In this section, we collect some results needed for the proof of Theorem \ref{main_theorem} and for our applications of Theorems \ref{p-nilpotence_theorem} and \ref{main_theorem}. 
	
	\begin{lemma}
		\label{BerkovichIsaacsLemma} 
		Let $H$ be an $s$-semipermutable subgroup of a group $G$. 
		\begin{enumerate}
			\item[(1)] (\cite[Lemma 3.1]{BerkovichIsaacs}) If $H$ is a $p$-subgroup of $G$ and $N$ is a normal subgroup of $G$, then $HN/N$ is $s$-semipermutable in $G/N$. 
			\item[(2)] (\cite[Lemma 3.2]{BerkovichIsaacs}) If $H$ is a $p$-subgroup of $G$ and if $N$ is a normal $p$-subgroup of $G$, then $H \cap N$ is normalized by $O^p(G)$.
			\item[(3)] (\cite[Lemma 3.3]{BerkovichIsaacs}) If $H \le K \le G$, then $H$ is $s$-semipermutable in $K$.
			\item[(4)] (\cite[Theorem A]{Isaacs}) If $H$ is a $p$-group, then $H^G$ is solvable, where $H^G$ denotes the normal closure of $H$ in $G$. 
		\end{enumerate}
	\end{lemma}
	
	\begin{lemma}
		\label{Yu2017} 
		Let $G$ be a group of order divisible by $p$, and let $P$ be a Sylow $p$-subgroup of $G$. Suppose that $P$ has a subgroup $D$ with $1 < |D| < |P|$ such that any subgroup of $P$ with order $|D|$ is $s$-permutable in $G$. If $P$ is a nonabelian $2$-group and $|D| = 2$, suppose moreover that any cyclic subgroup of $P$ with order $4$ is $s$-permutable in $G$. Then $G$ is $p$-supersolvable. 
	\end{lemma}
	
	\begin{proof}
		This follows from \cite[Theorem 1.1]{Yu}.	
	\end{proof} 
	
	Following Ballester-Bolinches et al. \cite{BB}, we say that a subgroup $H$ of a group $G$ is \textit{$c$-supplemented} in $G$ if there is a subgroup $K$ of $G$ such that $G = HK$ and $H \cap K \le H_G$, where $H_G$ is the core of $H$ in $G$. 
	
	\begin{lemma}
		\label{Asaad1}
		(\cite[Theorem 3.2 and Corollary 3.4]{Asaad1}) Let $P$ be a nontrivial normal $p$-subgroup of a group $G$. Suppose that there is a subgroup $D$ of $P$ with $1 \le |D| < |P|$ such that every subgroup of $P$ with order $|D|$ or $p|D|$ is $c$-supplemented in $G$. If $P$ is a nonabelian $2$-group and $|D| = 1$, suppose moreover that any cyclic subgroup of $P$ with order $4$ is $c$-supplemented in $G$. Then $P$ is contained in the supersolvable hypercentre $Z_{\mathcal{U}}(G)$ of $G$. 
	\end{lemma}
	
	\begin{lemma}
		\label{Asaad2} 
		(\cite[Theorem 3.5]{Asaad1}) Let $p$ be the smallest prime dividing the order of a group $G$, and let $P$ be a Sylow $p$-subgroup of $G$. Suppose that there is a subgroup $D$ of $P$ with $1 \le |D| < |P|$ such that every subgroup of $P$ with order $|D|$ or $p|D|$ is $c$-supplemented in $G$. If $P$ is a nonabelian $2$-group and $|D| = 1$, suppose moreover that any cyclic subgroup of $P$ with order $4$ is $c$-supplemented in $G$. Then $G$ is $p$-nilpotent.  
	\end{lemma}

	\section{Proof of Theorem \ref{main_theorem}} 
	\label{proof_main_theorem} 
	\begin{proof}[Proof of Theorem \ref{main_theorem}] Suppose that the theorem is not true, and let $G$ be a counterexample with minimal order. We will derive a contradiction in several steps. 
		
		\medskip
		(1) \textit{$G/N$ is $p$-supersolvable for any nontrivial normal subgroup $N$ of $G$.}
		
		Let $1 \ne N \trianglelefteq G$ and $\overline G := G/N$. Then $\overline G$ is $p$-solvable, $\overline P$ is a Sylow $p$-subgroup of $\overline G$, we have $\overline{P}' = \overline{P'} \le \overline{H} \le \overline{\Phi(P)} = \Phi(\overline P)$, and $\overline H$ is $s$-semipermutable in $\overline G$ by Lemma \ref{BerkovichIsaacsLemma} (1). Also $N_{\overline G}(\overline P) = \overline{N_G(P)}$ is $p$-supersolvable. The minimality of $G$ implies that $\overline G$ is $p$-supersolvable. 
		
		\medskip	
		(2) $O_{p'}(G) = 1$.
		
		Assume that $O_{p'}(G) \ne 1$. Then $G/O_{p'}(G)$ is $p$-supersolvable by (1), which implies that $G$ is $p$-supersolvable. This contradiction shows that $O_{p'}(G) = 1$. 
		
		\medskip
		(3) \textit{We have $\Phi(G) = 1$, and $O_p(G)$ is the unique minimal normal subgroup of $G$.}  
		
		Assume that $\Phi(G) \ne 1$. Then $G/\Phi(G)$ is $p$-supersolvable by (1). Applying \cite[Kapitel VI, Satz 8.6 a)]{Huppert}, we conclude that $G$ is $p$-supersolvable. This contradiction shows that $\Phi(G) = 1$. 
		
		Let $N_1, \dots, N_t$ be the distinct minimal normal subgroups of $G$. From \cite[Theorem 1.8.17]{Guo}, we see that $O_p(G) = N_1 \times \dots \times N_t$. To finish the proof of (3), it is enough to show that $t = 1$. 
		
		Assume that $t > 1$. By \cite[Theorem 1.3.7]{Guo}, $G \cong G/(N_1 \cap N_2)$ is isomorphic to a subgroup of $(G/N_1) \times (G/N_2)$. Since $G/N_1$ and $G/N_2$ are $p$-supersolvable by (1), we have that $(G/N_1) \times (G/N_2)$ is $p$-supersolvable. It follows that $G$ is $p$-supersolvable. This contradiction yields $t = 1$. 
		
		\medskip
		(4) $H \cap O_p(G) = 1.$
		
		Set $U := H \cap O_p(G)$. As $P' \le H$, we have $H \trianglelefteq P$. Hence, $U$ is the intersection of two normal subgroups of $P$, so $U \trianglelefteq P$ and thus $P \le N_G(U)$. By Lemma \ref{BerkovichIsaacsLemma} (2), we also have $O^p(G) \le N_G(U)$. Thus $G = PO^p(G) \le N_G(U)$ and hence $U \trianglelefteq G$. 
		
		Since $O_p(G)$ is minimal normal in $G$ by (3), we either have $U = 1$ or $U = O_p(G)$. Assume that $U = O_p(G)$. Then $O_p(G) \le H \le \Phi(P)$, and \cite[Kapitel III, Hilfssatz 3.3]{Huppert} implies that $O_p(G) \le \Phi(G)$, which contradicts (3). So we have $U = 1$, as required. 
		
		\medskip
		(5) \textit{Each maximal subgroup of $G$ contains $O_p(G)$ or $P'$.} 
		
		Let $M$ be a maximal subgroup of $G$ such that $O_p(G) \not\le M$. We have to show that $P' \le M$. 
		
		By (4), we have $(P \cap M)' \cap O_p(G) \le P' \cap O_p(G) \le H \cap O_p(G) = 1$. Thus 
		\begin{equation*}
			(P \cap M)' \cap O_p(G) = P' \cap O_p(G) = 1. 
		\end{equation*} 
		
		As $O_p(G) \not\le M$, we have $G = MO_p(G)$. So $P = P \cap G = P \cap MO_p(G) = (P \cap M)O_p(G)$. This implies that 
		\begin{equation*} 
			(P/O_p(G))' = (P \cap M)'O_p(G)/O_p(G) \cong (P \cap M)'/((P \cap M)' \cap O_p(G)) \cong (P \cap M)'.
		\end{equation*}  
		On the other hand, we have 
		\begin{equation*} 
			(P/O_p(G))' = P'O_p(G)/O_p(G) \cong P'/(P' \cap O_p(G)) \cong P'.
		\end{equation*}  
		So we have $(P \cap M)' \cong (P/O_p(G))' \cong P'$. As $(P \cap M)' \le P'$, it follows that $P' = (P \cap M)' \le M$. This completes the proof of (5). 
		
		\medskip
		(6) \textit{The final contradiction.} 
		
		Let $\mathscr{M}$ denote the set of all maximal subgroups $M$ of $G$ with $O_p(G) \not\le M$. We have $\mathscr{M} \ne \varnothing$, because otherwise $O_p(G) \le \Phi(G)$, which contradicts (3). Set 
		\begin{equation*}
			M^* := \bigcap_{M \in \mathscr{M}} M. 
		\end{equation*}
		Since $\mathscr{M}$ is invariant under conjugation with elements of $G$, we have $M^* \trianglelefteq G$.
		
		Let $M \in \mathscr{M}$. Since $O_p(G)$ is the unique minimal normal subgroup of $G$ by (3) and since $O_p(G)$ is not contained in $M$, we have $M_G = 1$. As $M^* \le M_G$, it follows that $M^* = 1$. 
		
		From (5), we see that each member of $\mathscr{M}$ contains $P'$. Therefore, we have $P' \le M^* = 1$, and so $P$ is abelian. 
		
		Since $G$ is $p$-solvable and $O_{p'}(G) = 1$, we have $C_G(O_p(G)) \le O_p(G)$ by \cite[Chapter 6, Theorem 3.2]{Gorenstein}. As $P$ is abelian, it follows that $P = O_p(G)$. Consequently, we have $G = N_G(O_p(G)) = N_G(P)$, and so $G$ is $p$-supersolvable by the hypotheses of the theorem. This final contradiction completes the proof.  
	\end{proof} 
	
	\section{A new proof of a special case of Theorem \ref{p-nilpotence_theorem}}
	\label{proof_special_case} 
	A special case of Theorem \ref{p-nilpotence_theorem} is the statement that a group $G$ with Sylow $p$-subgroup $P$ is $p$-nilpotent provided that $N_G(P)$ is $p$-nilpotent and that there is a subgroup $H$ of $P$ with $P' \le H \le \Phi(P)$ such that $H$ is $s$-semipermutable in $G$. In this section, we use Theorem \ref{main_theorem} to give a new proof of this fact, which is somewhat shorter than the proof of Theorem \ref{p-nilpotence_theorem}.
	
	\begin{theorem}
		Let $G$ be a group, and let $P$ be a Sylow $p$-subgroup of $G$. Suppose that there is a subgroup $H$ of $P$ such that $P' \le H \le \Phi(P)$ and such that $H$ is $s$-semipermutable in $G$. Suppose moreover that $N_G(P)$ is $p$-nilpotent. Then $G$ is $p$-nilpotent. 
	\end{theorem}
	
	\begin{proof}
		Suppose that the theorem is false, and let $G$ be a minimal counterexample.
		
		Using Lemma \ref{BerkovichIsaacsLemma} (1), we see that the hypotheses of the theorem carry over to $G/O_{p'}(G)$. Therefore, if $O_{p'}(G) \ne 1$, then $G/O_{p'}(G)$ is $p$-nilpotent by the minimality of $G$, which implies that $G$ is $p$-nilpotent. Thus $O_{p'}(G) = 1$. 
		
		Suppose that $PH^G = G$. Then $G/H^G \cong P/(P \cap H^G)$ is solvable, and $H^G$ is solvable by Lemma \ref{BerkovichIsaacsLemma} (4). Consequently, $G$ is solvable. By hypothesis, $N_G(P)$ is $p$-nilpotent and hence $p$-supersolvable. So Theorem \ref{main_theorem} implies that $G$ is $p$-supersolvable. Since $O_{p'}(G) = 1$, it follows from \cite[Lemma 2.1.6]{Products} that $P \trianglelefteq G$. So we have that $G = N_G(P)$ is $p$-nilpotent. This contradiction shows that $PH^G$ is a proper subgroup of $G$. 
		
		Using Lemma \ref{BerkovichIsaacsLemma} (3), we see that $PH^G$ satisfies the hypotheses of the theorem. So $PH^G$ is $p$-nilpotent by the minimality of $G$. Hence, $H^G$ is $p$-nilpotent. As $O_{p'}(H^G) \le O_{p'}(G) = 1$, it follows that $H^G$ is a $p$-group. Lemma \ref{BerkovichIsaacsLemma} (2) implies that $H = H \cap H^G$ is normalized by $O^p(G)$. Since $H$ is also normalized by $P$ and $G = PO^p(G)$, it follows that $H \trianglelefteq G$. 
		
		Since $P/H$ is abelian and $N_{G/H}(P/H) = N_G(P)/H$ is $p$-nilpotent, we have that $G/H$ is $p$-nilpotent by Burnside's $p$-nilpotency criterion \cite[Chapter 7, Theorem 4.3]{Gorenstein}. As $H \le \Phi(P)$, we have $H \le \Phi(G)$ by \cite[Kapitel III, Hilfssatz 3.3]{Huppert}. It follows that $G/\Phi(G)$ is $p$-nilpotent. Applying \cite[Kapitel VI, Hilfssatz 6.3]{Huppert}, we conclude that $G$ is $p$-nilpotent. This contradiction completes the proof. 
	\end{proof}

	\section{Applications}
	\label{applications} 
	In this section, we use Theorems \ref{p-nilpotence_theorem} and \ref{main_theorem} to generalize some recent results on pronormal subgroups and weakly $\mathscr{H}$-subgroups obtained in \cite{LiuYu} and \cite{Chen} and to shorten their proofs. 
	
	Recall that a subgroup $H$ of a group $G$ is said to be \textit{pronormal} in $G$ if $H$ and $H^g$ are conjugate in $\langle H, H^g \rangle$ for each $g \in G$. Examples of pronormal subgroups are normal subgroups, Sylow subgroups and Hall subgroups of solvable groups. 
	
	Recently, Liu and Yu \cite{LiuYu} proved the following theorem. 
	
	\begin{theorem}
		(\cite[Theorem 1.3]{LiuYu})
		\label{LiuYu} 
		Let $G$ be a group of order divisible by $p$ such that $(|G|,p-1) = 1$, and let $P$ be a Sylow $p$-subgroup of $G$. Suppose that $P' \trianglelefteq G$ and that $P$ has a subgroup $D$ with $1 < |D| < |P|$ such that any subgroup of $P$ with order $|D|$ is pronormal in $N_G(P)$. If $P$ is a nonabelian $2$-group and $|D| = 2$, suppose moreover that any cyclic subgroup of $P$ with order $4$ is pronormal in $N_G(P)$. Then $G$ is $p$-nilpotent. 
	\end{theorem}
	
	By \cite[Chapter I, Lemma 6.3 (d)]{DoerkHawkes}, a subgroup $H$ of a group $G$ is normal in $G$ provided that $H$ is both subnormal and pronormal in $G$. If $P$ is a Sylow $p$-subgroup of a group $G$, then any subgroup of $P$ is subnormal in $N_G(P)$ since any subgroup of $P$ is subnormal in $P$ and $P \trianglelefteq N_G(P)$. Consequently, a subgroup of $P$ is pronormal in $N_G(P)$ if and only if it is normal in $N_G(P)$. Therefore, when we replace the word “pronormal” by the word “normal” in Theorem \ref{LiuYu}, then what we obtain is just a reformulation of Theorem \ref{LiuYu}. 
	
	As is well-known, any normal subgroup of a group $G$ is $s$-permutable in $G$, but the converse is not true. In the next theorem, we take the hypotheses of Theorem \ref{LiuYu}, but we replace pronormality by $s$-permutability, and we also weaken the condition $P' \trianglelefteq G$. 
	
	\begin{theorem}
		\label{LiuYu_gen}
		Let $G$ be a group of order divisible by $p$ such that $(|G|,p-1) = 1$, and let $P$ be a Sylow $p$-subgroup of $G$. Assume that there is a normal subgroup $H$ of $P$ with $H \le \Phi(P)$ and $P/H = Z_{p-1}(P/H)$ such that $H$ is $s$-semipermutable in $G$. Suppose that $P$ has a subgroup $D$ with $1 < |D| < |P|$ such that any subgroup of $P$ with order $|D|$ is $s$-permutable in $N_G(P)$. If $P$ is a nonabelian $2$-group and $|D| = 2$, suppose moreover that any cyclic subgroup of $P$ with order $4$ is $s$-permutable in $N_G(P)$. Then $G$ is $p$-nilpotent. 
	\end{theorem}
	
	\begin{proof}
		By Theorem \ref{p-nilpotence_theorem}, it suffices to show that $N_G(P)$ is $p$-nilpotent. Lemma \ref{Yu2017} shows that $N_G(P)$ is $p$-supersolvable. We have $(|N_G(P)|,p-1) = (|G|,p-1) = 1$, and so \cite[Lemma 2.4 (4)]{ChangwenLi} 
		implies that $N_G(P)$ is $p$-nilpotent, as required.
	\end{proof}
	
	Note that Theorem \ref{LiuYu} is covered by Theorem \ref{LiuYu_gen}. In the next theorem, we again take the hypotheses of Theorem \ref{LiuYu}, but we remove the condition that $(|G|,p-1) = 1$ and add the condition that $G$ is $p$-solvable. We also weaken the condition that $P' \trianglelefteq G$, and we replace pronormality by $s$-permutability. 
	
	\begin{theorem}
		Let $G$ be a $p$-solvable group, and let $P$ be a Sylow $p$-subgroup of $G$. Assume that there is a subgroup $H$ of $P$ such that $P' \le H \le \Phi(P)$ and such that $H$ is $s$-semipermutable in $G$. Suppose that $P$ has a subgroup $D$ with $1 < |D| < |P|$ such that any subgroup of $P$ with order $|D|$ is $s$-permutable in $N_G(P)$. If $P$ is a nonabelian $2$-group and $|D| = 2$, suppose moreover that any cyclic subgroup of $P$ with order $4$ is $s$-permutable in $N_G(P)$. Then $G$ is $p$-supersolvable.  
	\end{theorem}
	
	\begin{proof}
		By Theorem \ref{main_theorem}, it suffices to show that $N_G(P)$ is $p$-supersolvable. But this follows immediately from Lemma \ref{Yu2017}. 
	\end{proof}
	
	Following Bianchi et al. \cite{Bianchi}, we say that a subgroup $H$ of a group $G$ is an \textit{$\mathscr{H}$-subgroup} of $G$ if $N_G(H) \cap H^g \le H$ for all $g \in G$. Following Asaad et al. \cite{AsaadHelielAlShomrani}, we say that a subgroup $H$ of a group $G$ is a \textit{weakly $\mathscr{H}$-subgroup} of $G$ if there is a normal subgroup $K$ of $G$ such that $G = HK$ and $H \cap K$ is an $\mathscr{H}$-subgroup of $G$.
	
	Recently, Chen et al. \cite{Chen} proved the following theorem.  
	
	\begin{theorem}
		\label{theorem_Chen} 
		(\cite[Theorem A]{Chen}) Let $p$ be the smallest prime dividing the order of a group $G$, and let $P$ be a Sylow $p$-subgroup of $G$. Suppose that $P' \trianglelefteq G$ and that every maximal subgroup of $P$ is a weakly $\mathscr{H}$-subgroup of $N_G(P)$. Then $G$ is $p$-nilpotent. 
	\end{theorem}
	
	By \cite[Theorem 6 (2)]{Bianchi}, a subgroup $H$ of a group $G$ is normal in $G$ provided that $H$ is both subnormal and an $\mathscr{H}$-subgroup of $G$. In particular, if $P$ is a Sylow $p$-subgroup of a group $G$, then a subgroup of $P$ is an $\mathscr{H}$-subgroup of $N_G(P)$ if and only if it is normal in $N_G(P)$. Consequently, a subgroup of $P$ is $c$-supplemented in $N_G(P)$ if it is a weakly $\mathscr{H}$-subgroup of $N_G(P)$. 
	
	In the next two results, we take the hypotheses of Theorem \ref{theorem_Chen}, but we replace “weakly $\mathscr{H}$-subgroup” by “$c$-supplemented”, and we also modify some other conditions.  
	
	\begin{theorem}
		\label{gen_Chen}
		Let $p$ be the smallest prime dividing the order of a group $G$, and let $P$ be a Sylow $p$-subgroup of $G$. Assume that there is a normal subgroup $H$ of $P$ with $H \le \Phi(P)$ and $P/H = Z_{p-1}(P/H)$ such that $H$ is $s$-semipermutable in $G$. Suppose that there is a subgroup $D$ of $P$ with $1 \le |D| < |P$ such that every subgroup of $P$ with order $|D|$ or $p|D|$ is $c$-supplemented in $N_G(P)$. If $P$ is a nonabelian $2$-group and $|D| = 1$, suppose moreover that any cyclic subgroup of $P$ with order $4$ is $c$-supplemented in $N_G(P)$. Then $G$ is $p$-nilpotent. 
	\end{theorem}
	
	\begin{proof}
		By Theorem \ref{p-nilpotence_theorem}, it suffices to show that $N_G(P)$ is $p$-nilpotent. Lemma \ref{Asaad2} shows that $N_G(P)$ is $p$-nilpotent.
	\end{proof}
	
	\begin{theorem}
		Let $P$ be a Sylow $p$-subgroup of a $p$-solvable group $G$. Assume that there is a subgroup $H$ of $P$ such that $P' \le H \le \Phi(P)$ and such that $H$ is $s$-semipermutable in $G$. Suppose that there is a subgroup $D$ of $P$ with $1 \le |D| < |P$ such that every subgroup of $P$ with order $|D|$ or $p|D|$ is $c$-supplemented in $N_G(P)$. If $P$ is a nonabelian $2$-group and $|D| = 1$, suppose moreover that any cyclic subgroup of $P$ with order $4$ is $c$-supplemented in $N_G(P)$. Then $G$ is $p$-supersolvable.  
	\end{theorem}
	
	\begin{proof}
		As a consequence of Lemma \ref{Asaad1}, every chief factor of $N_G(P)$ below $P$ is cyclic. Consequently, $N_G(P)$ is $p$-supersolvable. So $G$ is $p$-supersolvable by Theorem \ref{main_theorem}. 
	\end{proof}

	\noindent \textbf{Disclosure Statement.} The authors report there are no competing interests to declare.

\end{document}